\documentclass[11pt,letterpaper]{amsart}
\usepackage{amsmath}
\usepackage{amssymb}
\usepackage{amsthm}
\usepackage{hyperref}
\usepackage[noabbrev,capitalize]{cleveref}
\usepackage{mathrsfs}
\usepackage{mathtools}
\usepackage{slashed}
\usepackage{tikz-cd}
\usepackage[shortlabels]{enumitem}

\setenumerate{label=(\roman*)}

\def\sideremark#1{\ifvmode\leavevmode\fi\vadjust{\vbox to0pt{\vss
 \hbox to 0pt{\hskip\hsize\hskip1em
 \vbox{\hsize3cm\tiny\raggedright\pretolerance10000
 \noindent #1\hfill}\hss}\vbox to8pt{\vfil}\vss}}}

\newtheorem{theorem}{Theorem}[section]
\newtheorem{proposition}[theorem]{Proposition}
\newtheorem{lemma}[theorem]{Lemma}
\newtheorem{corollary}[theorem]{Corollary}

\theoremstyle{definition}

\theoremstyle{remark}
\newtheorem{remark}[theorem]{Remark}

\numberwithin{equation}{section}

\makeatletter
\@namedef{subjclassname@2020}{\textup{2020} Mathematics Subject Classification}
\makeatother

\begin{document}

\title[Do Carmo's problem in $\mathbb{R}^6$]{Do Carmo's problem for CMC hypersurfaces in $\mathbb{R}^6$}

\author{Jingche Chen}
\address{Department of Mathematical Sciences, Tsinghua University, 100084, Beijing, China}
\email{cjc23@mails.tsinghua.edu.cn}

\author{Han Hong}
\address{Department of Mathematics and statistics \\ Beijing Jiaotong University \\ Beijing \\ China, 100044}
\email{hanhong@bjtu.edu.cn}

\author{Haizhong Li}
\address{Department of Mathematical Sciences, Tsinghua University, 100084, Beijing, China}
\email{lihz@tsinghua.edu.cn}

\begin{abstract}
In this paper, we prove that complete noncompact constant mean curvature hypersurfaces in $\mathbb{R}^6$ with finite index must be minimal. This provides a positive answer to do Carmo's question in dimension $6$. The proof strategy is also applicable to $\mathbb{R}^4$ and $\mathbb{R}^5$, thereby providing alternative proofs for those previously resolved cases.
\end{abstract}

\maketitle

\section{introduction}
The classical Bernstein problem asks whether a complete minimal graph over Euclidean space $\mathbb{R}^{n+1}$ must necessarily be a hyperplane. Through a series of groundbreaking works by W. Fleming \cite{Flembing62}, De Giorgi \cite{De65}, F. Almgren \cite{Almgren66}, J. Simons \cite{Simons68}, and E. Bombieri, De Giorgi, and E. Giusti \cite{BDG69}, this problem has been fully resolved. Meanwhile, S. Chern \cite{cherntheorem} showed that complete constant mean curvature (CMC) graphs must be minimal. Indeed, if the mean curvature is nonzero, one can use the maximum principle to derive a contradiction by translating a sphere with the same mean curvature from above or below until it touches the graph (see Nelli \cite{Nelli-survey}).

The community is now interested in a stronger problem, known as the \textit{stable Bernstein problem}, which aims to classify complete noncompact, two-sided, stable minimal hypersurfaces. R. Schoen, L. Simon and S. Yau \cite{schoensimonyau1975} made significant progress by proving that such hypersurfaces in $\mathbb{R}^{n+1}$ for $n\leq 5$ must be hyperplanes under the assumption of Euclidean volume growth. Since then, efforts have been devoted to removing the volume growth assumption. In the early 1980s, D. Fischer-Colbrie and R. Schoen \cite{Fischer-Colbrie-Schoen-The-structure-of-complete-stable}, M. do Carmo and C.K. Peng \cite{doCarmo-Peng}, and Pogorelov \cite{Pogorelov-stable} independently proved that complete two-sided stable minimal surfaces in $\mathbb{R}^3$ must be planes.
Exciting progress has been made in recent years. In a series of works by O. Chodosh and C. Li \cite{chodoshliR4,chodoshliR4anisotropic}, O. Chodosh, C. Li, P. Minter and D. Stryker\cite{chodosh-li-minter-stryker-R5}, L. Mazet \cite{mazet}, this problem has been resolved up to dimension $n=5$. It is also worth noting that G. Catino, P. Mastrolia and A. Roncoroni \cite{catino} provided an alternative proof in dimension $n=3$. The only remaining unsolved case is stable minimal hypersurfaces in $\mathbb{R}^7$, although R. Schoen and L. Simon \cite{SchoenSimon1981Regularity}, C. Bellettini \cite{bellettini} have establised some partial results in this dimension.

We can also define stability for CMC hypersurfaces from a variational perspective, similar to minimal hypersurfaces: they are critical points of the area functional under compactly supported variations that preserve the enclosed volume. Let $M\rightarrow (X^{n+1},g)$ be an immersed hypersurface with constant mean curvature. It is called \textit{weakly stable} if
\[\int_M |\nabla\varphi|^2\geq \int_M (|A_M|^2+\operatorname{Ric}_X(\nu,\nu))\varphi^2\]
	for any $\varphi\in C_0^\infty(M)$ satisfying
	\[\int_M \varphi=0.\]
The corresponding Jacobi operator  is given by $$J=\Delta_M+|A_M|^2+\operatorname{Ric}_X(\nu,\nu).$$ 
We say that $M$ has \textit{finite index} if $\sup_{i}\operatorname{Ind}(\Omega_i)<\infty$, where $\{\Omega_i\}$ is an exhaustion of $M$ and $\operatorname{Ind}(\Omega_i)$ denotes the number of negative Dirichlet eigenvalues of $J$ on $\Omega_i$. 

In 1989, do Carmo \cite{docarmoquestionlecturenote} posed the following question: \textit{Is a complete noncompact, weakly stable (even finite index), constant mean curvature hypersurface of $\mathbb{R}^{n+1}$, $n\geq 3$, necessarily minimal?}   The question is natural. Because before it was posed, da Silveira \cite{daSilveira-Stability-of-complete}, F. L\'{o}pez and A. Ros \cite{lopezros} have confirmed it for $n=2$. Moreover, It is a generalization of Chern's result since complete CMC graph must be stable. H. Alencar and M. do Carmo \cite{alencardocarmopolynomialgrowth} showed that a complete CMC hypersurface in $\mathbb{R}^{n+1}$ with finite index and  polynomial volume growth must, in fact, be minimal. Much later, the answer was proven to be positive for $n=3,4$ by X. Cheng \cite{chengxufiniteindex} and independently by M. Elbert, B. Nelli and H. Rosenberg\cite{elbertnellirosenbergpams},  both using a generalized Bonnet-Meyer's argument. Additionally, under the assumption that the volume entropy of $M$ is zero, S. Ilias, B. Nelli, and M. Soret \cite{zeroentropy} provided a positive answer to Do Carmo's question in all dimensions. With conditions on the growth of total curvature, there are some results for stable CMC hypersurfaces in $\mathbb{R}^6$ (see \cite{alencar-docarmo-totalcurvature,docarmo-zhou-totalcurvature,ilias-nelli-soret-totalcurvature}). However, do Carmo's question remains open in higher dimensions without additional assumptions. Notably, resolving do Carmo's question is equivalent to proving that CMC hypersurfaces with nonzero mean curvature and finite index must be compact.

The goal of this paper is to answer do Carmo's question positively in dimension $n=5$.
\begin{theorem}\label{maintheorem}
    Complete noncompact constant mean curvature hypersurfaces in $\mathbb{R}^6$ with finite index must be minimal.
\end{theorem}

The assumption of finite index is necessary because there exist Delaunay surfaces, which have infinite index (due to their periodic nature) and nonzero mean curvature. Under the assumption of weak stability, the second author \cite[Theorem 6.1]{hong24} has shown that complete CMC hypersurfaces in $\mathbb{R}^6$ with nonzero mean curvature are compact. The case of finite index differs from the stable case because the hypersurface is only stable outside a compact subset. This can lead to significant challenges; for instance, obtaining a global Schoen-Yau type inequality on a CMC hypersurface with nonzero index is often difficult.

 As previously mentioned, the earlier method relies on a generalized Bonnet-Myers argument, which inherently imposes a dimension restriction. Our approach to the proof, however, differs from this and focuses on establishing the volume estimate of the geodesic ball of the hypersurface. Inspired by recent advancements in resolving stable Bernstein problems for minimal hypersurfaces, we employ $\mu$-bubbles, Sobolev inequality and volume comparison theorem to derive the volume estimate.
 
 To achieve this, let's assume that CMC hypersufaces under consideration have nonzero mean curvature. We first show that a Sobolev-type inequality holds for CMC hypersurfaces with finite index, and that the number of ends is finite (Theorem \ref{finiteend}). The Sobolev type inequality naturally holds outside a compact set but can be extended to the whole hypersurface. We then adopt the strategy outlined in \cite{chodoshliR4anisotropic, chodosh-li-stryker, mazet}, utilizing the $\mu$-bubble technique, but adapt it to CMC hypersurfaces with nonzero mean curvature. This allows us to show that CMC hypersurfaces with finite index have finite volume, i.e.,
 \[|B^M_R(p)|\leq C.\]
One of difficulties is that we lack of isoperimetric inequality for CMC hypersurface. Thus we can not directly bound the volume of a ball by the area of its boundary that consists of $\mu$-bubbles. However, the Sobolev-type inequality ensures that the volume of a ball is bounded by the volume of a compact region outside the ball. With this, we can use Bishop-Gromov volume comparison theorem of hypersurface to control the volume of a slice with fixed width in the ends.

 Finally, the finite volume is a contradiction since the hypersurface is noncompact CMC.  In fact, we also prove the result of Theorem \ref{maintheorem} in $\mathbb{R}^4$ and $\mathbb{R}^5$, thus providing a new idea compared to prior approaches. A similar idea was applied in \cite{hong24} for classifying CMC hypersurface with finite index in $\mathbb{H}^4.$

One particularly interesting and direct corollary of Theorem \ref{maintheorem}, which is worth noting, is the following maximum principle at infinity.
 \begin{corollary}
   A properly embedded complete CMC hypersurface in $\mathbb{R}^6$
  cannot lie on the mean convex side of another properly embedded complete CMC hypersurface with the same nonzero mean curvature.
 \end{corollary}

This is clear by the maximum principle if one of them is compact. If not, we can construct a compact area-minimizing CMC hypersurface with compact prescribed boundary on one of the CMC hypersurfaces, denoted by $M_1$. The other hypersurface $M_2$ serves as a barrier. Then, letting the prescribed boundary tend to infinity allows us to obtain a third stable complete noncompact CMC hypersurface, which is impossible by our main theorem.

Another interesting application of our result lies in the min-max theory of constant mean curvature hypersurfaces. Building on Theorem \ref{maintheorem}, the main result of \cite{Mazurowski-Zhou} can be extended to dimensions up to $n+1=6$. The key missing ingredient in this context is a complete classification of finite index CMC hypersurfaces with nonzero mean curvature.
\vskip.2cm

 The rest of paper is organized as follows. In section \ref{section2} we establish a Sobolev inequality for CMC hypersurfaces with finite index and demonstrate that such hypersurfaces have a finite number of ends. In section \ref{section3}, we outline the strategy of Chodosh-Li and others for CMC hypersurfaces. Section \ref{section4} is devoted for the proof of the main theorem. In the Appendix, we provide a detailed proof of extension result of Sobolev inequality.

 \subsection{Acknowledgements}
The authors would like to thank G. Carron for clarifying for us a point in \cite{carron_L^2_cohomologie}. We also thank Professor Xin Zhou and Professor Ao Sun for pointing out an application of our result in the min-max theory \cite{Mazurowski-Zhou}. We appreciate Professor Barbara Nelli for some comments on the first draft. The second author is supported by NSFC No. 12401058 and the Talent Fund of Beijing Jiaotong University No. 2024XKRC008. The first and third authors are supported by NSFC Grant No. 12471047.

\section{Finite index implies finite ends}\label{section2}
This section is devoted to proving that the number of ends of a CMC hypersurface in $\mathbb{R}^{n+1}$ with finite index must be finite. In fact, this result holds more generally in any $(n+1)$-dimensional Riemannian manifold satisfying certain curvature conditions. Hereafter, we only assume the two-sidedness of the CMC hypersurface if the mean curvature is zero, since this property is automatic for CMC hypersurfaces with nonzero mean curvature.

To begin, let us revisit the work of P. Li and J. Wang \cite{Li-Wang-finiteindex}, which establishes the finiteness of ends under certain conditions. Let $M$ be a complete noncompact Riemannian manifold, and let $\omega$ be an $L^2$-harmonic $1$-form on $M$.
Denote $h=|\omega|$, and let $H^1(L^2(M))$ represent the space of square-integrable harmonic $1$-forms. Using an algebraic lemma, Li and Wang derived the following key result.
\begin{theorem}[Li-Wang\cite{Li-Wang-finiteindex}]\label{liwang02-key theorem}
    Assume $M$ is a complete noncompact $n$-dimensional Riemannian manifold satisfying the following Sobolev type inequality

    \begin{equation}\label{sobolev type inequality}\left(\int_{M\setminus B_R(p)}(\phi h)^{\frac{2n}{n-2}}\right)^{\frac{n-2}{2}}\leq C\int_{M\setminus B_R(p)}|\nabla \phi|^2 h^2\end{equation}
    for some $R>0$, where $B_{R}(p)$ is a ball with radius $R$ centered at some $p\in M$, $C$ is a uniformly constant only depending on $M$. Then $\operatorname{dim} H^1(L^2(M))<\infty$. In particular, $M$ has finitely many ends. 
\end{theorem}

As a consequence, the number of ends of a minimal hypersurface with finite index in $\mathbb{R}^{n+1}$ is finite, owing to the existence of a Sobolev inequality on minimal hypersurfaces in $\mathbb{R}^{n+1}$ \cite{Li-Wang-finiteindex}. The primary objective of this section is to extend this result to CMC hypersurfaces. The approach essentially builds on their ideas.

We shall first show the following proposition.
\begin{proposition}\label{ricquass}
    Let $M^{n} \hookrightarrow (X^{n+1},g)$ be a complete noncompact constant mean curvature hypersurface. Assume that the sectional curvature of $X$ is nonnegative, then 
    \[
        \operatorname{Ric}(\omega^{\#},\omega^{\#}) \geq -\frac{\sqrt{n-1}}{2}|A|^2 |\omega|^2
    \]
    for any 1-form $\omega$ on $M$, where $\omega^{\#}$ is the dual vector field.
\end{proposition}
\begin{proof}
    Choose an orthonormal frame $\{e_{i}\}_{i=1}^{n}$ such that $h_{ij}=\lambda_i\delta_{ij}$, then assume that $\omega^{\#}= \sum_{i=1}^{n}a_{i}e_{i}$.
    According to the Gauss equation, we have
    \[
    \begin{aligned}
   \operatorname{Ric}(\omega^{\#},\omega^{\#})
    &=\sum_{i=1}^{n}a_{i}a_{j}(Hh_{ij}-\sum_{k=1}^{n}h_{ik}h_{jk})+\sum_{i,j,k=1}^na_ia_jR^X_{ikjk}\\
    &\geq\sum_{i=1}^{n}a_{i}^2(H\lambda_{i}-\lambda_{i}^2)\\
    &=\sum_{i=1}^{n}a_{i}^2\big(\sum_{j\neq i}\lambda_{i}\lambda_{j}\big)\\
    &\geq-\sum_{i=1}^{n}a_{i}^2\Big(\frac{\sqrt{n-1}}{2}\sum_{j\neq i}\big(\frac{\lambda_{i}^2}{n-1}+\lambda_{j}^{2}\big)\Big)\\
    &=-\frac{\sqrt{n-1}}{2}|A|^2|\omega|^2.
    \end{aligned}
    \]
    The proof is complete.
\end{proof}

\begin{proposition}\label{sobolevineq}
    Let $M^{n} \hookrightarrow (X^{n+1},g)$ $(n\geq 3)$ be a complete noncompact constant mean curvature hypersurface with finite index. Assume that the sectional curvature of $X$ is nonnegative, the asymptotic volume ratio is nonzero, and $X$ has bounded geometry. Then, $M$ satisfies \begin{equation}
        \left(\int_M|f|^{\frac{2n}{n-2}}\right)^{\frac{n-2}{n}}\leq C(n,\theta)\int_M |\nabla f|^2
    \end{equation}
    for any function $f$ compactly supported in $M$, where the constant $C(n,\theta)$ depends on the dimension $n$ and the asymptotic volume ratio defined by
    \[\theta=\lim_{R\rightarrow \infty}\frac{V(B_R(p))}{R^{n+1}}.\]
\end{proposition}
\begin{proof}
    It follows from \cite{simonbrendle-michael} that Michael-Simon-Sobolev inequality holds, i.e.,
    \begin{equation}
    c(n,\theta)\Big(\int_{M}|\varphi|^{\frac{n}{n-1}}\Big)^{\frac{n-1}{n}}\leq \int_{M} |\nabla \varphi| + |\varphi H|
    \end{equation}
    for any $\varphi\in C_{0}^{\infty}(M)$, where $C(n,\theta)$ is a constant that depends on the dimension $n$ and  $\theta$ of $X$.
    Taking $\varphi=f^{\frac{2(n-1)}{n-2}}$, we have
    \begin{equation}\label{sobo2}
    c(n,\theta)\Big(\int_{M}|f|^{\frac{2n}{n-2}}\Big)^{\frac{n-1}{n}}\leq \int_{M} \frac{2(n-1)}{n-2}|f|^{\frac{n}{n-2}}|\nabla f| + |f|^{\frac{2(n-1)}{n-2}}|H|.
    \end{equation}
    According to the H\"{o}lder inequality, we have
    \begin{equation}\label{es1}
    \begin{aligned}
    \int_{M}|f|^{\frac{n}{n-2}}|\nabla f| 
    &\leq \Big(\int_{M}|\nabla f|^2\Big)^{\frac{1}{2}}\Big(\int_{M}|f|^{\frac{2n}{n-2}}\Big)^{\frac{1}{2}}.    
    \end{aligned}
    \end{equation}
    
 On the other hand, since $M$ has finite index, it follows from \cite{Fischer-Colbrie-On-complete-minimal} that there exists a compact subset $\Omega$ such that $M\setminus \Omega$ is strongly stable. Without lost of generality, we assume that $\Omega$ is contained in a ball $B_R(p)$ with radius $R$ and centered at $p\in \Omega.$ Then for any compactly supported function $\psi\in C^1_c(M\setminus B_R)$,
 \begin{equation*}
     \int_{M\setminus B_R}|A|^2\psi^2\leq \int_{M\setminus B_R}|\nabla\psi|^2.
 \end{equation*}
 Then for $f\in C^1_c(M\setminus B_R)$
    \begin{equation}\label{es2}
    \begin{aligned}
    \int_{M}|f|^{\frac{2(n-1)}{n-2}}|H| 
    &\leq \Big(\int_{M}f^2H^2\Big)^{\frac{1}{2}}\Big(\int_{M}|f|^{\frac{2n}{n-2}}\Big)^{\frac{1}{2}}\\
    &\leq \Big(\int_{M}n f^2|A|^2\Big)^{\frac{1}{2}}\Big(\int_{M}|f|^{\frac{2n}{n-2}}\Big)^{\frac{1}{2}}\\
    &\leq \Big(\int_{M}n|\nabla f|^2\Big)^{\frac{1}{2}}\Big(\int_{M}|f|^{\frac{2n}{n-2}}\Big)^{\frac{1}{2}}.
    \end{aligned}
    \end{equation}
    Plugging \eqref{es1}and \eqref{es2} into \eqref{sobo2}, we obtain the inequality
    \begin{equation}
        \left(\int_M|f|^{\frac{2n}{n-2}}\right)^{\frac{n-2}{n}}\leq C(n,\theta)\int_M |\nabla f|^2.
    \end{equation}
    Finally, the result follows from the Proposition \ref{carron-proposition}.
    
\end{proof}
\begin{proposition}\label{ends are nonparabolic}
    Let $M^{n} \hookrightarrow (X^{n+1},g)$ $(n\geq 3)$ be a complete noncompact, constant mean curvature hypersurface with finite index. Assume that the sectional curvature of $X$ is nonnegative and $X$ is of bounded geometry, then each end of $M$ is nonparabolic.  In particular, the number of ends of $M$ is bounded by the dimension of $H^1(L^2(M))+1.$
\end{proposition}

\begin{proof}
    According to \cite[Corollary 4]{Li-Wang-finiteindex} and Proposition \ref{sobolevineq},  each end of $M$ must either have finite volume or be non-parabolic. It is well-known that each end of complete noncompact CMC hypersurface in Riemannian manifold with bounded geometry have infinite volume. Thus, each end of $M$ is nonparabolic. The last assertion follows from \cite[Corollary 2]{Li-Wang-finiteindex}: the number of nonparabolic ends of $M$ is bounded from above by $\operatorname{dim} H^1(L^2(M))+1.$
\end{proof}

\begin{theorem}\label{finiteend}
    Let $M^{n} \hookrightarrow (X^{n+1},g)$ $(3\leq n \leq 6)$ be a complete noncompact constant mean curvature hypersurface with finite index. Assume that the sectional curvature of $X$ is nonnegative and $X$ is of bounded geometry and the asymptotic volume ratio  of $X$ is nonzero, then $\operatorname{dim} H^{1}(L^{2}(M)) < \infty$. In particular, $M$ has finite ends.
\end{theorem}
\begin{proof}
    Recall $\omega$ is a harmonic 1-form and denote $h=|\omega|\in L^{2}(M)$.
    By Bochner formula and Proposition \ref{ricquass}, we have
    \begin{equation}\label{es3}
    \begin{aligned}
        h\Delta h 
        &=\text{Ric}(\omega,\omega)+ |\nabla\omega|^2-|\nabla h|^2\\
        &\geq -\frac{\sqrt{n-1}}{2}|A|^2h^2+ |\nabla\omega|^2-|\nabla h|^2.
    \end{aligned}
    \end{equation}
    Using the improved Kato's inequality (see \cite{Li-Wang-finiteindex})
    \[
        |\nabla\omega|^2 \geq \frac{n}{n-1}|\nabla h|^2,
    \]we have
    \begin{equation}\label{ineq4}    
    h\Delta h \geq -\frac{\sqrt{n-1}}{2}|A|^2h^2+\frac{|\nabla h|^2}{(n-1)}.
    \end{equation}
     Multiplying \eqref{ineq4} with $\phi^2$ where $\phi$ is a function compactly supported in $M\setminus B_{R_0}(p)$ and integrating on $M\backslash B_{R_0}(p)$ yield
    \begin{equation}\label{equ1}
        \frac{\sqrt{n-1}}{2}\int_{M\backslash B_{R_0}(p)}\phi^2 |A|^2 h^2 + \int_{M\backslash B_{R_0}(p)}\phi^2 h \Delta h \geq \frac{1}{n-1}\int_{M\backslash B_{R_0}(p)} |\nabla h|^2\phi^2.
    \end{equation}
    Using integration by parts and the stability inequality, we can estimate the terms on the left hand side as follows.
    \begin{equation}\label{es4}
    \begin{aligned}
    &\frac{\sqrt{n-1}}{2}\int_{M\backslash B_{R_0}(p)}\phi^2 |A|^2 h^2 + \int_{M\backslash B_{R_0}(p)}\phi^2 h \Delta h \\
    &\leq \frac{\sqrt{n-1}}{2}\int_{M\backslash B_{R_0}(p)}|\nabla(\phi h)|^2 - 2\int_{M\backslash B_{R_0}(p)}\phi h \langle \nabla\phi,\nabla h\rangle - \int_{M\backslash B_{R_0}(p)}\phi^2|\nabla h|^2\\
    &=\frac{\sqrt{n-1}}{2}\int_{M\backslash B_{R_0}(p)}|\nabla \phi|^2h^2 + \big(\sqrt{n-1}-2\big)\int_{M\backslash B_{R_0}(p)}\phi h \langle \nabla \phi , \nabla h\rangle \\
    &+ \big(\frac{\sqrt{n-1}}{2}-1\big)\int_{M\backslash B_{R_0}(p)}\phi^2|\nabla h |^2.    
    \end{aligned}
    \end{equation}
    Plugging \eqref{es4} into \eqref{equ1} and using H\"{o}lder inequality, we obtain
    \begin{equation}
    \begin{aligned}
        &\big(\frac{1}{n-1}+1-\frac{\sqrt{n-1}}{2}\big)
        \int_{M\backslash B_{R_0}(p)}\phi^2|\nabla h|^2\\
        &\leq \frac{\sqrt{n-1}}{2}\int_{M\backslash B_{R_0}(p)}|\nabla \phi|^2h^2 + \big(\sqrt{n-1}-2\big)\int_{M\backslash B_{R_0}(p)}\phi h \langle \nabla \phi , \nabla h\rangle\\
        &\leq \big(\frac{\sqrt{n-1}}{2}+\frac{1}{\epsilon}\big|\frac{\sqrt{n-1}}{2}-1\big|)\int_{M\backslash B_{R_0}(p)}h^2|\nabla \phi|^2\\
        &\ \ \ +\epsilon\big|\frac{\sqrt{n-1}}{2}-1\big| \int_{M\backslash B_{R_0}(p)}\phi^2 |\nabla h|^2.
    \end{aligned}
    \end{equation}
    For $\forall\  3\leq n\leq 6$, we have that $\big(\frac{1}{n-1}+1-\frac{\sqrt{n-1}}{2}\big)>0$. Choosing $\epsilon$ small enough, we have
    \begin{equation}\label{es5}
    \begin{aligned}
    \int_{M\backslash B_{R_0}(p)}\phi^2|\nabla h|^2\leq C(n)\int_{M\backslash B_{R_0}(p)}h^2|\nabla \phi|^2.
    \end{aligned}
    \end{equation}
    By Proposition \ref{sobolevineq} and \eqref{es5}, we have
    \begin{equation}
        \Big(\int_{M}|\phi h|^{\frac{2n}{n-2}}\Big)^{\frac{n-2}{n}} \leq C(n,\theta) \int_{M\backslash B_{R_0}(p)}h^2|\nabla \phi|^2.
    \end{equation}
    Thus, it follows from Theorem \ref{liwang02-key theorem} that the dimension of $H^1(L^2(M))$ is finite. Furthermore, the number of ends is finite as shown in Proposition \ref{ends are nonparabolic}.
\end{proof}
\begin{remark}
    One should compare Theorem \ref{finiteend} with the work of P. Li and J. Wang \cite[Theorem 4.1]{Li-Wang-nonnegatively-curved-manifold}. They proved that a properly immersed minimal hypersurface with finite index in a nonnegatively curved Riemannian manifold must have finite ends. While properness is essential in their result, this assumption can be relaxed if the ambient manifold has bounded geometry, which is precisely the setting we consider. However, their proof relies heavily on the minimality of the hypersurface, specifically the subharmonicity of the Busemann function, a property that does not directly extend to constant mean curvature hypersurfaces. As a result, it is unclear how their approach could be adapted to the CMC case, leaving open the need for new techniques in this setting.
\end{remark}

    It is interesting to explore what weaker curvature conditions, compared to nonnegative sectional curvature of ambient manifolds, can imply that the number of ends of a constant mean curvature hypersurface (or even a minimal hypersurface) with finite index is finite.
    In a recent work by the second author and G. Wang \cite{hong-wang-spectral-splitting}, we prove that an index-zero CMC hypersurface $M$ in a Riemannian manifold $X^{n+1}$ with nonnegative biRic curvature (defined below) and dimension $n\leq 5$ has at most two ends.  This is also implicitly showed in recent works of Antonelli-Pozzetta-Xu\cite{antonelli-pozzetta-xu} and Catino-Mari-Mastrolia-Roncoroni \cite{catino-Mari-Mastrolia-Roncoroni} independently.

Let us conclude this section by emphasizing that the Euclidean spaces $\mathbb{R}^{n+1}$ (for $3 \leq n \leq 6$) satisfy the necessary assumptions outlined in Theorem \ref{finiteend}.

\section{Conformal metric and construction of $\mu$-bubble}\label{section3}
In this section, we restrict our discussion to Euclidean space $\mathbb{R}^{n+1}$. Let $M$ be a complete noncompact CMC hypersurface with finite index in $\mathbb{R}^{n+1}$. Assume that $M \subset \mathbb{R}^{n+1}$ contains the origin after rigid motions. We first recall the definitions of the biRic curvature and the $\alpha$-biRic curvature. Under the induced metric on $M$, the biRic curvature operator is defined as
\[
\operatorname{biRic}(x, y) = \operatorname{Ric}_M(x, x) + \operatorname{Ric}_M(y, y) - R_M(x, y, x, y),
\]
for any two orthonormal vectors $x, y \in T_pM$, where $R_M$ is the sectional curvature operator of $M$. The $\alpha$-biRic curvature operator is defined as
\[
\operatorname{biRic}_\alpha(x, y) = \operatorname{Ric}_M(x, x) + \alpha \operatorname{Ric}_M(y, y) -\alpha R_M(x, y, x, y),
\]
which reduces to the biRic curvature operator when $\alpha = 1$. Then, the biRic curvature and $\alpha$-biRic curvature are defined as
\[
\operatorname{biRic}(p) = \min\{\operatorname{biRic}(x, y) : x, y \text{ are orthonormal in } T_pM\}
\]
and
\[
\operatorname{biRic}_\alpha(p) = \min\{\operatorname{biRic}_\alpha(x, y) : x, y \text{ are orthonormal in } T_pM\},
\]
respectively. The study of biRic curvature on stable minimal hypersurfaces dates back to Shen and Ye \cite{shenyingyerugang}.

Since $M$ has finite index, $M \setminus B_{R_0}(0)$ is stable for a sufficiently large radius $R_0$. Hereafter the ball denotes intrinsic ball of $M$. For simplicity, denote $B_{R_0}(0)$ by $B_{R_0}$. Let 
\[
    a_{n} = \begin{cases}
        \frac{111}{100}, & n=5\\
        \frac{11}{10}, & n=4\\
        1, &n=3
    \end{cases}
    \quad \text{and} \quad  
    \alpha_{n} = \begin{cases}
        \frac{93}{100}, &n=5\\
        1, & n=4\\
        1, &n=3.
    \end{cases}
\]
which will be clear later.
We have the following result.
\begin{proposition}\label{lower bound for alpha biric}
    Let $M^n\hookrightarrow \mathbb{R}^{n+1}$ $(3\leq n \leq 5)$ be a complete noncompact hypersurface with constant mean curvature $H\neq 0$ and finite index . Then there exist $\delta_n >0$ such that 
    \[
    \int_{M}|\nabla \varphi|_{g}^2-\frac{1}{a_n}(\delta_n-\operatorname{biRic}_{\alpha_{n}})\varphi^2 \geq 0 
    \]
    for any $\varphi \in C_{c}^{1}(M\setminus B_{R_0})$.
\end{proposition}
\begin{proof}
    For convenience, we denote $a_{n}$ by $a$ and also denote $\alpha_{n}$ by $\alpha$.
    Firstly as discussed above, there exists a constant $R_0>0$ such that $M\setminus B_{R_0}$ is stable. The constant $R_0$ depends on the dimension $n$ and the index.  Then
    \begin{equation}
    \int_{M}|\nabla \varphi|_{g}^2 \geq \int_{M} |A|^2 \varphi^2 .
    \end{equation}
    for $\varphi\in C_c^1(M\setminus B_{R_0})$.
    Thus,
    
    \begin{equation}\label{inequality involving biric}
    \begin{aligned}
        &\int_M|\nabla\varphi|^2_{g}+\frac{1}{a}\operatorname{biRic}_{\alpha}\varphi^2 \geq\int_M \left(|A|^2+\frac{1}{a}\operatorname{biRic}_{\alpha}\right)\varphi^2 .
    \end{aligned}
    \end{equation}

 Let $\{e_{i}\}_{1\leq i\leq n}$ be orthonormal basis with respect to the metric $g$. Now we compute the $\text{biRic}_{\alpha}$. By Gauss equation, we have
    \begin{equation}\label{biric}
    \begin{aligned}
        &\text{biRic}_{\alpha}(e_{1},e_{2}) 
        =\sum_{i=2}^{n}R_{1i1i}+\alpha \sum_{j=3}^{n}R_{2j2j}\\
        &=\sum_{i=2}^{n}(A_{11}A_{ii}-A_{1i}^2)+\alpha \sum_{j=3}^{n}(A_{22}A_{jj}-A_{2j}^2)
    \end{aligned}
    \end{equation}

   We take 
    \begin{equation}   \label{matrix_A}
    P=
    \begin{pmatrix}
     a & \frac{1}{2} & \frac{1}{2} & \cdots & \frac{1}{2} \\
     \frac{1}{2} & a & \frac{\alpha}{2} & \cdots & \frac{\alpha}{2} \\
     \frac{1}{2} & \frac{\alpha}{2} & a &\cdots & 0\\
     \vdots & \vdots & \vdots & \ddots & \vdots\\
     \frac{1}{2} & \frac{\alpha}{2} & 0 & \cdots & a
    \end{pmatrix}, \qquad x=\begin{pmatrix}
        A_{11}\\
        A_{22}\\
        A_{33}\\
        \vdots\\
        A_{nn}
    \end{pmatrix}
    \end{equation}
    
    We claim $a,\ \alpha$ satisfy:
    \begin{equation}\label{cond1}
        \left\{\begin{array}{cc}
        2a\geq \alpha \\
        2a\geq 1\\
        P>0.
        \end{array}\right.
    \end{equation}
    In fact,
    
    \textbf{Case 1:} $n=5, a=\frac{111}{100}, \alpha = \frac{93}{100}$ satisfy \eqref{cond1}. Since $\operatorname{det}(P)= \frac{1}{4} a^2 (4 a^3 + 3 \alpha - a (4 + 3 \alpha^2)) =\frac{11586754647}{40000000000}>0$ and $\operatorname{det}(C)= 
    a^4 - \frac{3 a^2 \alpha^2}{4} =\frac{287535177}{400000000}>0$ where $C$ is the submatrix of $P$ with column 2 to column $n$ and row 2 to row $n$, then $P$ is positive definite.

\vskip.2cm
    \textbf{Case 2:} $n=4, a=\frac{11}{10}, \alpha = 1$ satisfy \eqref{cond1}. Since $\operatorname{det}(P)=-\frac{3 a^2}{4} + a^4 + \frac{a \alpha}{2} - \frac{a^2 \alpha^2}{2}=\frac{627}{1250}>0$ and $\operatorname{det}(C)= a^3 - \frac{a \alpha^2}{2}=\frac{781}{1000}>0$ where $C$ is the submatrix of $P$ with column 2 to column $n$ and row 2 to row $n$, then $P$ is positive definite.
    
\vskip.2cm

    \textbf{Case 3:} $n=3, a=1, \alpha = 1$ satisfy \eqref{cond1}. Since $\operatorname{det}(P)=-\frac{a}{2} + a^3 + \frac{\alpha}{4} - \frac{a\alpha^2}{4}=\frac{1}{2}>0$ and $\operatorname{det}(C)=a^2-\frac{\alpha^2}{4}=\frac{3}{4}>0$ where $C$ is the submatrix of $P$ with column 2 to column $n$ and row 2 to row $n$, then $P$ is positive defintie.
    
   \vskip.2cm

    Since $P>0$, then $\exists\ \delta(P)>0$, such that $P\geq \delta(P)I$.
    Thus, we have 
    \begin{equation}\label{estimate about biric}
    \begin{aligned}
        &a|A|^2+ \operatorname{biRic}_\alpha(e_{1},e_{2}) \\
        &=a|A|^2+\sum_{i=2}^{n}(A_{11}A_{ii}-A_{1i}^2)+\alpha \sum_{j=3}^{n}(A_{22}A_{jj}-A_{2j}^2)\\
        &=a\sum_{i=1}^{n}A_{ii}^2+\sum_{i=2}^{n}A_{11}A_{ii}+\alpha\sum_{j=3}^{n}A_{22}A_{jj}+a\sum_{i\neq j}A_{ij}^2-\sum_{i=2}^{n}A_{1i}^2-\alpha\sum_{j=3}^{n}A_{2j}^2\\ 
        &\geq x^{T}Px \geq \delta(P) \sum_{i=1}^{n}A_{ii}^2\geq \frac{\delta(P)}{n} H^2>0
    \end{aligned}
    \end{equation}
    where we use $2a\geq\alpha$ and $2a\geq 1$ to control non-diagonal elements.
   
    Thus, it follows from \eqref{inequality involving biric} and \eqref{estimate about biric} that there exist $\delta_{n} =\frac{\delta(P)}{n}H^2  >0$, such that 
    \[
    \int_{M}|\nabla \varphi|_{g}^2-\frac{1}{a}(\delta-\text{biRic}_{\alpha})\varphi^2  \geq 0 
    \]
    for any $\varphi \in C_{c}^{1}(M\setminus B_{R_0})$.
\end{proof}

Using Proposition \ref{lower bound for alpha biric} and similar arguments from \cite{Fischer-Colbrie-Schoen-The-structure-of-complete-stable}, there exists a positive function $\omega$ on $M\setminus B_{R_0}$ satisfying certain partial differential equation.

\begin{proposition}
    Assume $(M\setminus B_{R_0},g)$ and $a_{n}, \ \alpha_{n},\ \delta_{n}$ are set as above. Then there is a positive function $\omega$, such that
    \begin{equation}\label{equation omega}
\begin{cases}
    -a_{n}\Delta\omega+\operatorname{biRic}_{\alpha_{n}} \cdot\omega=\delta_{n}\omega, & \text{on}\ M\setminus B_{R_0}\\
    \omega=1, & \text{on}\ \partial (M\setminus B_{R_0}).
\end{cases}\end{equation}
\end{proposition}
\begin{proof}
    For convenience, we denote $a_{n}$ by $a$ and also denote $\alpha_{n}$ by $\alpha$. Consider an exhaustion $\Omega_{1}\subset \Omega_{2}\subset\cdots\subset M\backslash B_{R_{0}}$ by pre-compact regions with smooth boundary and each $\Omega_{i}$ contains $\partial B_{R_{0}}$. Fix $\Omega_{i}$, we consider the following equation
    \begin{equation}\label{equation stable}
    \begin{cases}
        -\Delta \omega_{i} + \frac{1}{a}(\operatorname{biRic}_{\alpha}-\delta)\omega_{i} = 0, & \text{on}\ \Omega_{i}\\
        \omega_{i} = 0 , & \text{on}\ \partial\Omega_{i}\backslash\partial B_{R_{0}}\\
        w_{i} = 1, & \text{on} \ \partial B_{R_{0}}.
    \end{cases}
    \end{equation}
    If the following homogeneous equation \eqref{equation stable 2} has only trivial solution, then the existence of equation \eqref{equation stable} is guaranteed  by Ferdholm alternative.
    \begin{equation}\label{equation stable 2}
    \begin{cases}
        -\Delta v_{i} + \frac{1}{a}(\operatorname{biRic}_{\alpha}-\delta)v_{i} = 0, & \text{on}\ \Omega_{i}\\
        v_{i} = 0 , & \text{on}\ \partial\Omega_{i}.
    \end{cases}
    \end{equation}
    In fact, multiply equation \eqref{equation stable 2} with $v_{i}$ and integrate on $\Omega_{i}$, we have
    \begin{equation}
    \begin{aligned}
        0=\int_{\Omega_{i}} |\nabla v_{i}|^2+\frac{1}{a}(\delta-\operatorname{biRic_{\alpha}})v_{i}^2\geq \lambda_{1}(\Omega_{i})\int_{\Omega_{i}}v_{i}^2.
    \end{aligned}    
    \end{equation}
    By the domain monotonicity formula $\lambda_{1}(\Omega_{i})>0$, then $v_{i}=0$.

    Next, we need to show $\omega_{i}> 0$. Since the strong Maximum principle, it suffices to show that $\omega_{i}\geq 0$. If it fails, then $V=\{x\in\Omega_{i}|\ \omega_{i}<0\}\neq \emptyset$. Consider the negative part of $\omega_{i}$ denoted by $\omega_{i}^{-}$. Then $-\Delta \omega_{i}^{-} + \frac{1}{a}(\operatorname{biRic}_{\alpha}-\delta)\omega_{i}^{-}=0$ on $V$ and $\omega_{i}^{-} = 0$ in $\partial V$. According to $\lambda_{1}(V)>0$, we have $\omega_{i}^{-}=0$ on $V$. It is a contradiction with $V\neq \emptyset$.
    
    According to the Harneck inequality, on any compact set $K$ containing $\partial B_{R_{0}}$ we have
    \[
        \sup_{K} \omega_{i} \leq C(K) \inf_{K}\omega_{i}\leq C(K)
    \]
    where $\partial B_{R_{0}}\subset K$ implies $\inf_{K}w_{i}\leq 1$. Thus, $\omega_{i}$ have a universal upper bound $C(K)$ on any compact set $K$ containing $\partial B_{R_{0}}$. By standard Schauder estimate, $|\omega_i|_{C^{2,\gamma}(K)}\leq C(K)$. Using Arzelà-Ascoli lemma and a diagonal argument, up to a subsequence, we have that $\omega_{i}\rightarrow \omega_{\infty}$ in $C_{loc}^{2}(M\backslash B_{R_{0}})$. Then $\omega_{\infty}$ is $C^{2}$ and satisfies the equation \eqref{equation omega}. 
    
    Finally, we show $\omega_{\infty}>0$. Since $\omega_{i}>0$, it follows that $\omega_{\infty}\geq 0$. According to the strong Maximum principle, we have either $\omega_{\infty}>0$ or $\omega_{\infty}=0$, the latter is impossible because $\omega_{\infty}|_{\partial B_{R_{0}}}=1$.
\end{proof}
We introduce the warped $\mu$-bubble. Let $N$ be a domain with boundary contained in $(M\setminus B_{R_0},g)$ and $\partial N = \partial_{-}N\cup \partial_{+}N$. Choose $h$ to be a smooth function on $\mathring{N}$ such that $h\rightarrow \pm\infty$ on $\partial_{\pm}N$. Choose a Caccioppolli set $\Omega_{0}$ with $\partial\Omega_{0}\subset\mathring{N}$ and $\partial_{+}N\subset \Omega_{0}$. We consider the following functional
\[
    \mathcal{A}(\Omega) = \int_{\partial^*\Omega}w^{a_{n}}d\mathcal{H}^{n-1}- \int_{M}(\chi_{\Omega}-\chi_{\Omega_{0}})hw^{a_{n}}d\mathcal{H}^{n}.
\]
We call $\Sigma=\partial^*\Omega$ a warped $\mu$-bubble, where $\Omega$ minimizes $\mathcal{A}$ over all sets of finite perimeter containing $\partial_+ N$. Let $\eta$ be the unit normal vector of $\Sigma$ pointing outside of $\Omega$. 

The following proposition is essentially from \cite[Section 4]{mazet}. For the sake of convenience of readers, we include it here and check that $a_{n}$, $\alpha_{n}$ above satisfy the necessary condition.

\begin{proposition}\label{bubble}
    Assume $(M\setminus B_{R_0},g)$ and $a_{n}, \ \alpha_{n}$ are set as above. Let $N\subset E_{i}$ be a domain with boundary where $E_{i}$ is one of the connected component of $M\setminus B_{R_0}$. If there exist a point $p\in N$ s.t. $\operatorname{dist}_{g}(p,\partial N)\geq C(a_{n},\delta_{n})>0$. Then there is a relatively open subset $\Omega$ containing $\partial N$ with $\Omega \subset B_{C(a_{n},\delta_{n})}(\partial N)$  such that on each connected component, denoted by $\Sigma$, of $\partial\Omega\setminus \partial N$, there holds
    \begin{equation}\label{ric1}
    \begin{aligned}
    \frac{4}{4-a_{n}}\int_{\Sigma}|\nabla_{\Sigma} \varphi|^2 \geq \int_{\Sigma}(\frac{\delta_{n}}{2}-\alpha_{n} \operatorname{Ric}_{\Sigma})\varphi^2    
    \end{aligned}
    \end{equation}
    for any $\varphi \in C^{1}(\Sigma)$, where $B_{C(a,\delta)}(\partial N)$ means the $C(a,\delta)$-neighborhood of $\partial N$ with respect to the metric $g.$
\end{proposition}

\begin{proof}
For convenience, we denote $a_{n}$ by $a$ and also denote $\alpha_{n}$ by $\alpha$. Computing the first variation and the second variation of $\mathcal{A}$, we obtain
\[
0=\frac{d}{dt}\mathcal{A}(\Omega_{t})|_{t=0} = \int_{\Sigma}(H_{\Sigma}+a(d\ln \omega) (\eta)-h)\omega^{a}\varphi
\]
and
\[
\begin{aligned}
    0\leq \frac{d^2}{dt^2}\mathcal{A}(\Omega_{t})|_{t=0} = \int_{\Sigma}w^{a}\Big(-\varphi \Delta_{\Sigma} \varphi-(|B|^2+\text{Ric}_{N}(\eta,\eta))\varphi^2-a\omega^{-2}d\omega(\eta)^2\varphi^2\\+a\omega^{-1}\nabla_{N}^{2}\omega(\eta,\eta)\varphi^2-a\omega^{-1}\langle\nabla_{\Sigma}\omega,\nabla_{\Sigma}\varphi\rangle\varphi-dh(\eta)\varphi^2\Big)
\end{aligned}
\]
where $B$ is the second fundamental form of $\Sigma$.
Using integration by parts and $\nabla^2_{N}\omega(\eta,\eta) = \Delta_{N}\omega - \Delta_{\Sigma} \omega -H_{\Sigma} d\omega(\eta)$, 
we have
\[
\begin{aligned}
    0\leq\int_{\Sigma}w^{a}\Big(|\nabla_{\Sigma}\varphi|^2-(|B|^2+\text{Ric}_{N}(\eta,\eta))\varphi^2-a\omega^{-2}d\omega(\eta)^2\varphi^2\\+a\omega^{-1}(\Delta_{N}\omega-\Delta_{\Sigma}\omega-H_{\Sigma}d\omega(\eta))\varphi^2-dh(\eta)\varphi^2\Big)
\end{aligned}
\]

Taking $\varphi = \omega^{-\frac{a}{2}}\psi$
and using that $\omega$ satisfies the equation \eqref{equation omega} and the Gauss equation, we get the following inequality
\[
\begin{aligned}
    &\int_{\Sigma}|\nabla_{\Sigma}\psi|^2+a\omega^{-1}\psi\langle\nabla_{\Sigma}\omega,\nabla_{\Sigma}\psi\rangle-(a-\frac{a^2}{4})\psi^2\omega^{-2}|\nabla_{\Sigma}\omega|^2
    \geq \int_{\Sigma}\psi^2\Big(\delta-\alpha \text{Ric}_{\Sigma}\\
    &+|B|^2+\alpha H_{\Sigma} B_{11}-\alpha \sum_{j=1}^{n-1}B_{1j}^2+a(d \ln \omega(\eta))^2+a H_{\Sigma} (d\ln \omega)(\eta)+ a dh(\eta)\Big).
\end{aligned}
\]
Then we use $\omega^{-1}\psi\langle\nabla_{\Sigma}\omega,\nabla_{\Sigma}\psi\rangle\leq \frac{1}{4-a}|\nabla_{\Sigma}\psi|^2+\frac{4-a}{4}\psi^2\omega^{-2}|\nabla_{\Sigma}\omega|^2$ to control the left hand side, we have
\[
\begin{aligned}
    \frac{4}{4-a}\int_{\Sigma}|\nabla_{\Sigma}\psi|^2
    &\geq \int_{\Sigma}\psi^2\Big(\delta-\alpha \text{Ric}_{\Sigma}+|B|^2+\alpha H_{\Sigma} B_{11}-\alpha \sum_{j=1}^{n-1}B_{1j}^2 \\
    &+a(d \ln \omega(\eta))^2+a H_{\Sigma} (d\ln \omega)(\eta)+ a dh(\eta)\Big).
\end{aligned}
\]

If we can show 
\begin{equation}\label{inequation about B}
    |B|^2 + \alpha H_{\Sigma} B_{11} -\alpha \sum_{j=1}^{n-1}B_{1j}^2 + a(d \ln \omega(\eta))^2 + a H_{\Sigma} (d\ln \omega)(\eta) \geq \beta h^2 ,
\end{equation}
then
\begin{equation}\label{inequalityonbubblewith-h}
     \frac{4}{4-a}\int_{\Sigma}|\nabla_{\Sigma}\psi|^2 \geq \int_{\Sigma}\psi^2(\frac{\delta}{2}-\alpha \text{Ric}_{\Sigma})+\psi^2(\frac{\delta}{2}+\beta h^2+adh(\eta)).
\end{equation}   
It is not hard to construct $h$ such that $\frac{\delta}{2}+\beta h^2+adh(\eta)\geq 0$. Then we finish the proof of the Proposition. It suffices to show \eqref{inequation about B}.
Using $H_{\Sigma} = h - a(d\ln \omega)(\eta)$, it is equivalent to show 
\[
|B|^2+\alpha H_{\Sigma}B_{11}-\alpha \sum_{j=1}^{n-1}B_{1j}^2+\frac{1}{a}(H_{\Sigma}-h)^2+H_{\Sigma}(h-H_{\Sigma})\geq \beta h^2.
\]
Denote $\Phi_{ij} = B_{ij} - \frac{H_{\Sigma}}{n-1}\delta_{ij}$. Then it suffices to show 
\[
\begin{aligned}
    \frac{1}{n-1}H_{\Sigma}^2+\frac{n-1}{n-2}\Phi_{11}^2+\frac{\alpha}{n-1}H_{\Sigma}^2+\alpha H_{\Sigma}\Phi_{11}-\alpha(\frac{1}{n-1}H_{\Sigma}+\Phi_{11})^2\\+\frac{1}{a}(H_{\Sigma}-h)^2+H_{\Sigma}(h-H_{\Sigma})\geq \beta h^2
\end{aligned}
\]
where we use $|\Phi|^2\geq \frac{n-1}{n-2}\Phi_{11}^2$ because $\Phi$ is trace free. We also use $\alpha\leq2$ to control non-diagonal elements. Denote   
\begin{equation}\label{cond2}
G=\begin{pmatrix}
    \frac{1}{n-1}+\frac{\alpha}{n-1}-\frac{\alpha}{(n-1)^2}+\frac{1}{a}-1 & \frac{\alpha}{2}\sqrt{\frac{n-2}{n-1}}(1-\frac{2}{n-1}) & \frac{1}{2}-\frac{1}{a} \\
    \frac{\alpha}{2}\sqrt{\frac{n-2}{n-1}}(1-\frac{2}{n-1}) & 1-\alpha\frac{n-2}{n-1} & 0\\
    \frac{1}{2}-\frac{1}{a} & 0 & \frac{1}{a}-\beta 
\end{pmatrix} .
\end{equation}
The above inequality is a quadratic form in $(H_{\Sigma},\sqrt{\frac{n-1}{n-2}}\Phi_{11},h)$ w.r.t. $G$. Thus it is equal to show $G>0$. 
In fact, 

\textbf{Case 1:} $n=5, a=\frac{111}{100}, \alpha = \frac{93}{100},\ \beta = \frac{1}{22}$. In this case, we have $1-\alpha\frac{n-2}{n-1}=\frac{121}{400} >0$ and
\[
    \operatorname{det}(G) = \frac{113639}{130240000}>0.
\]

\textbf{Case 2:} $n=4, a=\frac{11}{10}, \alpha=1,\ \beta = \frac{1}{22}$. In this case,  $1-\alpha\frac{n-2}{n-1}=\frac{1}{3}>0$ and
\[
    \operatorname{det}(G) = \frac{15}{242}>0.
\]

\textbf{Case 3:} $n=3, a=1, \alpha=1,\ \beta = \frac{1}{22}$. In this case, we have  $1-\alpha\frac{n-2}{n-1}=\frac{1}{2}>0$ and
\[
    \operatorname{det}(G) = \frac{41}{176}>0.
\]
Hence $G$ is positive definite. Thus, we are able to construct a warped $\mu$-bubble in a domain $N$ by using the weight function $\omega$. Moreover, the warped $\mu$-bubble has a spectral Ricci curvature lower bound. 
\end{proof}

\section{Proof of the main theorem}\label{section4}
In this section, we provide a proof of the main theorem. First, we need the following Lemma.
\begin{lemma}[\cite{chodoshliR4}]\label{boundary}
    Suppose that $M$ has $k$ ends and $H^{1}_{c}(M;\mathbb{R})$ has finite dimension.
    Consider an exhaustion $\Omega_{1}\subset \Omega_{2}\subset\cdots\subset M$ by pre-compact regions with smooth boundary. If each component of $M\backslash\Omega_{i}$ is unbounded. Then for $i$ sufficiently large, $\partial \Omega_{i}$ has $k$ components. 
\end{lemma}
In fact, the dimension of $H_c^1$ of CMC hypersurfaces in Euclidean space with finite index is finite.
\begin{proposition}\label{finite_H_1}
    Consider a complete noncompact CMC immersion $M^n\hookrightarrow \mathbb{R}^{n+1} $ with finite index. Then $H^{1}_{c}(M;\mathbb{R})$ has finite dimension.
\end{proposition}
\begin{proof}
    According to Proposition \ref{sobolevineq} and \cite[proposition 2.11]{Carron-L2-harmonic}, we know that 
    \[
    H_{c}^{1}(M) \longrightarrow H^{1}(L^{2}(M))
    \]
    is injective.
    Combining Theorem \ref{finiteend}, we can obtain the conclusion.
\end{proof}

    We also need the fact that $M$ has $\text{Ric}\geq -(n-1)Cg$ for a positive constant $C$. Actually, we can show that the second fundamental form of $M$ is uniformly bounded from above. Then the lower bound for Ricci curvature follows from the Gauss equation. The estimate on $|A|$ can be obtained by a standard point-picking argument, references include \cite{brianwhitelecturenote, bellettini-chodosh-neshan}.
    \begin{proposition}\label{ric estimate below}
        Let $ E_0$ be a stable end of $M^n$ $(n\leq 5)$, then for any compact subset $K \subset E_0$, 
    \begin{equation}\label{A_estimate}
    \sup_{K}\sup_{x\in K} |A_K|(x)\min\{1,\operatorname{dist}_K(x,\partial K)\}\leq C
    \end{equation}
    for  a constant $C$ depending only on the mean curvature. In particular, $|A_{E_0}|$ is uniformly bounded from above.
    \end{proposition}
    \begin{proof}
    If the assertion fails, we can find a sequence of  compact two-sided stable CMC immersions $K_i\rightarrow \mathbb{R}^{n+1}$ such that \[\sup_{x\in K_i} |A_{K_i}|(x)\min\{1,\operatorname{dist}_{K_i}(x,\partial K_i)\}\rightarrow \infty\]
    as $i$ tends to infinity. We can choose the point $q_i$ such that
    \[|A_{K_i}|(q_i)\min\{1,\operatorname{dist}_{K_i}(q_i,\partial K_i)\}\rightarrow \infty.\] Then, we have 
    \[|A_{K_i}|(q_i)\rightarrow \infty
    \quad \text{and}\quad |A_{K_i}|(q_i)\operatorname{dist}_{K_i}(q_i,\partial K_i)\rightarrow \infty. 
    \]
    We denoted $r_{i} = \min\{ |A_{K_{i}}|^{-\frac{1}{2}}(q_{i}),\operatorname{dist}_{K_{i}}(q_{i},\partial K_{i}) \}\rightarrow 0$.
    By the definition, we have $r_{i}\leq \operatorname{dist}_{K_{i}}(q_{i},\partial K_{i})$. Thus, $B_{r_{i}}(q_{i})\subset K_{i}$.
    Since $B_{r_{i}}(q_{i})$ is a compact set, then we can choose $p_{i}\in B_{r_{i}}(q_{i})$ such that 
    \[
    |A_{K_{i}}|(p_{i})\operatorname{dist}_{K_{i}}(p_{i},\partial B_{r_{i}}(q_{i})) = \max_{x\in B_{r_{i}}(q_{i})}|A_{K_{i}}|(x)\operatorname{dist}_{K_{i}}(x,\partial B_{r_{i}}(q_{i})).
    \]
    Then we have
    \begin{equation}\label{p_{i}max}
    \begin{aligned}
    |A_{K_{i}}|(p_{i})\operatorname{dist}_{K_{i}}(p_{i},\partial B_{r_{i}}(q_{i}))
    & \geq |A_{K_{i}}|(q_{i})\operatorname{dist}_{K_{i}}(q_{i},\partial B_{r_{i}}(q_{i}))\\
    &\geq |A_{K_{i}}|(q_{i}) r_{i}\\
    &\geq \min\{ |A_{K_{i}}|^{\frac{1}{2}}(q_{i}),|A_{K_{i}}|(q_{i})\operatorname{dist}_{K_{i}}(q_{i},\partial K_{i} \}\\
    &\rightarrow \infty.
    \end{aligned}
    \end{equation}
    We also verify that $|A_{K_{i}}|(p_{i})\rightarrow \infty$. Indeed, $\operatorname{dist}_{K_{i}}(p_{i},\partial B_{r_{i}}(q_{i}))\leq r_{i}\leq |A_{K_{i}}|^{-\frac{1}{2}}(q_{i})\rightarrow 0$. Thus, by\eqref{p_{i}max} we have $|A_{K_{i}}|(p_{i})\rightarrow \infty$. 
    Now we denote $ R_{i}:=\text{dist}_{{K}_{i}}(p_{i},\partial B_{r_{i}}(q_{i}))>0$.
    We claim that
    \begin{equation}\label{claim1}
    |A_{{K}_{i}}|(p_{i})\text{dist}_{{K}_{i}}(p_{i},\partial B_{R_{i}}(p_{i})) = \max_{x\in B_{R_{i}}(p_{i})} |A_{{K}_{i}}|(x)\text{dist}_{{K}_{i}}(x,\partial B_{R_{i}}(p_{i})).
    \end{equation}
    If the claim fails, then $\exists\ t_{i}\in B_{R_{i}}(p_{i})$ such that 
    \[
        |A_{{K}_{i}}|(t_{i})\text{dist}_{{K}_{i}}(t_{i},\partial B_{R_{i}}(p_{i})) >  |A_{{K}_{i}}|(p_{i})\text{dist}_{{K}_{i}}(p_{i},\partial B_{R_{i}}(p_{i})).
    \]
    Then we have
    \begin{equation*}
    \begin{aligned}
            |A_{{K}_{i}}|(t_{i})\text{dist}_{{K}_{i}}(t_{i},\partial B_{R_{i}}(p_{i})) 
            &> |A_{{K}_{i}}|(p
            _{i})\text{dist}_{{K}_{i}}(p_{i},\partial B_{R_{i}}(p_{i}))\\
            &= |A_{{K}_{i}}|(p_{i})\text{dist}_{{K}_{i}}(p_{i},\partial B_{r_{i}}(q_{i}))\\
            &\geq |A_{{K}_{i}}|(t_{i})\text{dist}_{{K}_{i}}(t_{i},\partial B_{r_{i}}(q_{i})).        
    \end{aligned}
    \end{equation*}
    Thus, we have
    \[
    \text{dist}_{{K}_{i}}(t_{i},\partial B_{R_{i}}(p_{i})) > \text{dist}_{{K}_{i}}(t_{i},\partial B_{r_{i}}(q_{i})).
    \] This is a contradiction to the inclusion of the two balls.

    Then we move $K_i$ such that $p_i$ is the origin and scale $K_i$ by $\lambda_{i}:=|A_{K_i}|(p_i)$ to obtain new immersions, denoted by $\tilde{K}_{i}$, i.e. 
    \[
        \tilde{K_{i}} = \lambda_{i}(K_{i}-p_{i}).
    \]
    Now we can locally estimate the norm of the second fundamental form of $\tilde{K}_{i}$. Actually, for any fix $s>0$ and $\forall\ x\in B^{\tilde{K}_{i}}_{s}(0)$, we have
    \[
    \begin{aligned}
        |A_{\tilde{K}_{i}}|(x)
        &\leq |A_{\tilde{K}_{i}}|(0)\frac{\operatorname{dist}_{\tilde{K}_{i}}(0,\partial B^{\tilde{K}_{i}}_{\lambda_{i}R_{i}}(0))}{\operatorname{dist}_{\tilde{K}_{i}}(x,\partial B^{\tilde{K}_{i}}_{\lambda_{i}R_{i}}(0))}\\
        &\leq \frac{\lambda_{i}R_{i}}{\lambda_{i}R_{i}-s}\\
        &\leq C
    \end{aligned}
    \]
    for $\lambda_{i}R_{i}$ large enough.
    In above inequality, the first inequality is due to \eqref{claim1} and $|A_{K_{i}}|(p_{i})\operatorname{dist}_{K_{i}}(p_{i},\partial B_{r_{i}}(q_{i}))$ is scaling invariant. The finial inequality is because of $\lambda_{i}R_{i}\rightarrow \infty$ from \eqref{p_{i}max}.
    
    Since $|A_{\tilde{K}_i}|$ is locally bounded and using higher order elliptic estimates, up to a subsequence, we will obtain a limiting complete noncompact hypersurface $M_{\infty}$ that is stable in $\mathbb{R}^{n+1}$. Since $H_{\tilde{K}_{i}}=\frac{H}{|A_{K_{i}}|(p_{i})}\rightarrow 0$, $M_{\infty}$ is a minimal hypersurface in $\mathbb{R}^{n+1}$. By stable Bernstein theorem $(n\leq 5)$, we know that $M_\infty$ is a hyperplane. This contradicts to the fact that $|A_{M_{\infty}}|(0)=1.$ 
    Thus we prove\eqref{A_estimate}.  
    
    In particular, 
    \[
    |A_{E_{0}}|\leq \max\{\sup_{\text{dist}(x,\partial E_{0})\leq 1}|A|(x),C\}.
    \] Thus, $|A_{E_0}|$ is uniformly bounded from above.
    \end{proof}
    
Now we are ready to show the proof of the theorem \ref{maintheorem}.
\begin{proof}[Proof of Theorem \ref{maintheorem}]
    Assume that $p=0\in M$. Since $M$ is a CMC hpyersurface of finite index, there exists $ R_{0}>0$ such that $M\backslash B_{R_0}(p)$ is a stable CMC hypersurface. It follows from Theorem \ref{finiteend} that the number of ends of $M$ is finite, then choose $\rho>R_0$ large enough such that $M\setminus B_{\rho}(p)$ has $k$ connected unbounded components. Denote the unbounded connected component of $M\backslash B_{\rho}(p)$ by $\bigcup_{i=1}^k E_{i}$. For each $i,$ by Proposition \ref{bubble}, we have a domain $\Omega_{i}$ satisfying $\partial E_{i}\subset \Omega_{i}\subset B_{C(a_{n},\delta_{n})}(\partial E_{i})$. We denote $E_{i}\backslash \Omega_{i} = \tilde{E}_{i}\cup \tilde{S}_{i}$, where $\tilde{E}_{i}$ is the unbounded connected component of $E_i\setminus \Omega_i$ and $\tilde{S}_{i}$ is the union of bounded components of $E_i\setminus \Omega_i$. Finally, we denote $\partial \tilde{E}_{i}=\cup_j \Sigma_{ij}$, 
    where each $\Sigma_{ij}$ is a closed  hypersurface and the index $j$ denotes the number of components. Notice that no a prior control on $j$ is known. On each $\Sigma_{ij}$, there holds
    \begin{equation}\label{inequalityonbubble}
        \frac{4}{4-a_{n}}\int_{\Sigma_{ij}}|\nabla_{\Sigma_{ij}} \varphi|^2 \geq \int_{\Sigma_{ij}}(\frac{\delta_{n}}{2}-\alpha_{n} \operatorname{Ric}_{\Sigma_{ij}})\varphi^2. 
    \end{equation}
    
    We construct $\Omega_{\rho} = B_{\rho}(p)\cup(\cup_i\Omega_{i})\cup S \cup (\cup_j\tilde{S}_{j} )$, where $S$ is the set that contains all bounded connected components of $M\backslash B_{\rho}(p)$. Then we have that $\partial \Omega_{\rho} = \cup_{i,j} \Sigma_{ij}$ and $\{\Omega_{\rho}\}_{\rho}$ is an exhaustion sequence of $M$. Moreover, all the components of $M\setminus \Omega_\rho$ are unbounded. Thus, it follows from Proposition \ref{finite_H_1} and Lemma \ref{boundary} that
    \[
        \#\{ \Sigma_{ij}\} = k
    \]
    for $\rho$ large enough. That is, we can denote $\{\Sigma_{ij}\}=\{\Sigma_1,\cdots,\Sigma_k\}.$
    
    Now we would like to obtain a uniform area upper bound for $\Sigma_{i}$.
    
    \textbf{Case 1:} $n=5$, $\frac{4}{(4-a_{n})\alpha_{n}} = \frac{40000}{26877}<\frac{3}{2}=\frac{n-2}{n-3}$.
    
    \textbf{Case 2:} $n=4$, $\frac{4}{(4-a_{n})\alpha_{n}} = \frac{40}{29}<\frac{2}{1}=\frac{n-2}{n-3}$.
    
    In both cases, above inequalities and \eqref{inequalityonbubble} allow us to apply the area estimate of Antonelli-Xu\cite[Theorem 1]{antonelli-xu} under the metric $g$ and obtain
    \[
        |\Sigma_{i}|_{g} \leq C\
    \]
    for a universal constant $C$.
    
    \textbf{Case 3:} $n=3$, $a_{n}=1$, $\alpha_{n}=1$. In this case, we can directly apply \cite[lemma 6.1]{chodoshliR4anisotropic} such that
   \[
        |\Sigma_{i}|_{g} \leq C\
    \]
    for a universal constant $C$.

    Now we consider a decomposition as $\tilde{E}_i\backslash B_{C(a_{n},\delta_{n})}(\partial E_{i}) = \hat{E}_{i}\cup \hat{S}_{i}$ where $\hat{E}_{i}$ is the unbounded connected component of $\tilde{E}_i\backslash B_{C(a_{n},\delta_{n})}(\partial E_{i})$ and $\hat{S}_{i}$ is the union of bounded components of $\tilde{E}_i\backslash B_{C(a_{n},\delta_{n})}(\partial E_{i}) $.
    We construct $\hat{\Omega}_{\rho} = B_{\rho+C(a,\delta)}(p) \cup S\cup(\cup_{i} \tilde{S}_{i})\cup(\cup_{j} \hat{S}_{j}) $. Then $M\setminus \hat{\Omega}_\rho = \cup_{i}\hat{E}_{i}$.
    We denote $\hat\Omega_{i} =\{x\in \hat{E}_{i}\ |\ \text{dist}_{g}(x,\partial{\hat{E}_{i}})\leq C(a_{n},\delta_{n}) \} $. 
    Take a cut off function $f_{\hat\Omega_{\rho}}=\eta(\frac{1}{C(a_{n},\delta_{n})}d(x,\hat{\Omega}_{\rho}))$, where 
    \[\eta=\begin{cases}
        1-x, & x\in[0,1]\\
        0, & x\in (1,+\infty).
    \end{cases}\]
    Plugging $f_{\hat\Omega_{\rho}}$ in Proposition \ref{sobolevineq} yields 
    \begin{equation}\label{volume inequality}
         |\hat{\Omega}_{\rho}|^{\frac{n-2}{n}}\leq C(n,\theta,a_{n},\delta_{n})\sum_{i}|\hat\Omega_{i}|.
    \end{equation}
    
    In what follows we shall use the volume comparison theorem of hypersurfaces (see \cite{Wei_Xu_Zhang}[lemma 5.3] for instance) to control the volume of $\hat{\Omega}_{i}$. 
    Firstly, for fixed $i\in\{1,\cdots,k\}$, we show that $\Sigma_{i}$ separates $\partial E_{i}$ and $\partial \hat E_{i}$. In fact, by the definition, $\Sigma_{i}$ separates  $E_{i}\backslash \hat{E}_{i}$ into two connected components. Thus, for any continuous curve $\gamma : [0,1]\rightarrow {E}_{i}\backslash \hat{E}_{i}$ with $\gamma(0)\in \partial E_{i}$, $\gamma(1)\in \partial \hat{E}_{i}$, we have $\gamma\cap \Sigma_{i} \neq \emptyset$.

    From the separating property of $\Sigma_{i}$, we know that, for any $s\geq \rho+C(a_{n},\delta_{n})$, 
    \begin{equation}\label{separating}
        \Theta_{i,s} := \frac{1}{s}\exp_{p}^{-1}((\partial B_{s}(p)\cap (\hat\Omega_{i}))\backslash \text{Cut}(p))\subset F_{i}(\Sigma_{i}\backslash\text{Cut}(p))
    \end{equation}
    where we define $F_{i}:\Sigma_{i}\backslash\text{Cut}(p)\rightarrow \mathbb{S}_{p}(M)$ as $F_{i}(q) = \frac{\exp_{p}^{-1}(q)}{|\exp_{p}^{-1}(q)|_{g(p)}}$ and $\mathbb{S}_p(M)$ is the unit tangent space of $M$ at $p$.

    We denote $r:F_{i}(\Sigma_{i}\backslash\text{Cut}(p))\rightarrow \mathbb{R}^{+}$ such that $\varphi(\theta)=\exp_{p}(r(\theta)\theta)$ is the first time to touch $\Sigma_{i}$. Thus, $J_{\varphi}(\theta)$ as the jacobian determinant of $\varphi(\theta)$ is the volume form of $\Sigma_{i}$. We also denote by $J(s,\theta)d\theta$ the volume form of $\partial B_{s}(p)$ and the geodesic sphere with radius $s$ in space form $R^{n}(-C)$ is denoted by $\partial B^{-C}(s)$.  
    Then for any $\rho + C(a_{n},\delta_{n})\leq s\leq \rho +  2C(a_{n},\delta_{n})$, we have
    \begin{equation}\label{estimate of volume form}
    J_{\varphi}(\theta) \geq J(r(\theta),\theta) \geq \frac{|\partial B^{-C}(r(\theta))|}{|\partial B^{-C}(s)|}J(s,\theta)\geq\frac{|\partial B^{-C}(\rho)|}{|\partial B^{-C}(s)|}J(s,\theta) .
    \end{equation}
    In the above, the first inequality is due to \cite{Wei_Xu_Zhang}[lemma 5.2]. 
    The second inequality is a consequence of the Bishop-Gromov volume comparison theorem, combined with Proposition \ref{ric estimate below} (see the preceding discussion) and \eqref{separating}. The final inequality holds because  $\Sigma_{i}\subset E_{i}\backslash \hat{E}_{i}$.
    
    Then, ignoring the measure zero set in the integral, we have
    \begin{equation}
    \begin{aligned}
    |\Sigma_{i}| 
    &\geq \frac{|\partial B^{-C}(\rho)|}{|\partial B^{-C}(s)|}\int_{\Theta_{i,s}}J(s,\theta)d\mathcal{H}^{n-1}(\theta)\\
    &= \frac{|\partial B^{-C}(\rho)|}{|\partial B^{-C}(s)|} |(\partial B_{s}(p)\cap \hat{\Omega}_{i})-\text{Cut}(p)|.    
    \end{aligned}
    \end{equation}
    
    According to the coarea formula and $|\partial B^{-C}(r)| \approx e^{(n-1)r} $, we get
    \begin{equation}\label{volume estimate of hat omega}
    \begin{aligned}
    |\hat{\Omega}_{i}| 
    &= \int_{\rho+C(a_{n},\delta_{n})}^{\rho+2C(a_{n},\delta_{n})}|(\partial B_{s}(p)\cap \hat{\Omega}_{i})-\text{Cut}(p)| ds 
    \\
    &\leq \int_{\rho+C(a_{n},\delta_{n})}^{\rho+2C(a_{n},\delta_{n})}C\frac{e^{(n-1)s}}{e^{(n-1)\rho}}|\Sigma_{i}|ds\\
    &=C(e^{(n-1)2C(a_{n},\delta_{n})}-e^{(n-1)C(a_{n},\delta_{n})})|\Sigma_{i}|\\
    &\leq C(a_{n},\delta_{n},n).
    \end{aligned}
    \end{equation}

    Then it follows from \eqref{volume inequality} and \eqref{volume estimate of hat omega} that
    \begin{equation}
    \begin{aligned}
    |B_{\rho}(p)|
    &\leq|B_{\rho}(p)\cup S|\\
    &\leq |\hat\Omega_{\rho}|\\
    &\leq C \left(\sum_{i}|\hat\Omega_{i}|\right)^{\frac{n}{n-2}}\\
    &\leq k^{\frac{n}{n-2}}C.
    \end{aligned}
    \end{equation}
    The upper bound is independent of $\rho$. It means that $M$ has finite volume. This is impossible since $M$ is a noncompact complete CMC hypersurface in Euclidean space that has infinite volume. 
\end{proof}

Let us conclude this section by discussing extensions of results to $\delta$-stable CMC hypersurfaces in $\mathbb{R}^6$. The method used in this paper can also be applied to prove results for $\delta$-stable CMC hypersurfaces. A CMC hypersurface is said to be $\delta$-stable if it satisfies
\[
\int_{M} \delta |A|^2 f^2 \leq \int_{M} |\nabla f|^2
\]
for any $f \in C_0^\infty(M)$ and $\delta \leq 1$. We claim that for $0.84 \leq \delta \leq 1$, complete noncompact $\delta$-stable constant mean curvature hypersurfaces in $\mathbb{R}^6$ must be minimal. The only change required in the proof is to replace the matrix $A$ by $\tilde{A} = A + (\delta - 1)a I$ in \eqref{matrix_A}. It is not difficult to see that $\tilde{A}$ increases with $\delta$. Therefore, we only need to verify that for $\delta = 0.84$, we can choose $a = 1.149$, $\alpha = 0.936$, and a sufficiently small $\beta$ to satisfy all the conditions. Let us remark that in \cite{hong-li-wang}, we studied $\delta$-stable minimal hypersurfaces in $\mathbb{R}^6$. For $0.811 \leq \delta \leq 1$, we established the Euclidean volume growth of $\delta$-stable minimal hypersurfaces in $\mathbb{R}^6$ (see \cite[Remark 1.11]{hong-li-wang}). Furthermore, when $\delta > 15/16 \approx 0.9375$, it is proved that $\delta$-stable minimal hypersurfaces must be hyperplanes (see \cite[Theorem 1.7]{hong-li-wang}). Therefore, we can conclude that noncompact $\delta$-stable ($\delta \in (15/16, 1]$) CMC hypersurfaces must also be hyperplanes.

\section{Appendix}\label{section5}

We state the proposition regarding the extension Sobolev inequality of Carron. The proof is basically same as the proof of \cite[Proposition 2.4]{carron_L^2_cohomologie}. For the sake of convenience of readers, we include it here.
\begin{proposition}\cite[Proposition 2.5]{carron_L^2_cohomologie}\label{carron-proposition}
    Suppose $(M,g)$ is a complete noncompact Riemannian manifold with infinite volume. If there is a compact set $K\subset M$ satisfying the following sobolev inequality
    \begin{equation*}
        \Big(\int_{M\backslash K} |u|^{\frac{2p}{p-2}}dx \Big)^{\frac{p-2}{p}}\leq 
        C\int_{M\backslash K} |du|^2dx
    \end{equation*}
    for any $u\in C_{0}^{\infty}(M\backslash K)$, then $M$ also satisfies sobolev inequality
    \begin{equation*}
        \Big(\int_{M} |u|^{\frac{2p}{p-2}}dx \Big)^{\frac{p-2}{p}}\leq 
        C(M)\int_{M} |du|^2dx 
    \end{equation*}
    for any $u\in C_{0}^{\infty}(M)$.
\end{proposition}
\begin{proof}
    Take a compact set $\tilde{K}$ containing $K$ and a cut-off function $\theta$ satisfies $\theta = 1$ on $K$ and $\theta = 0$ outside $\tilde{K}$. For any $u\in C_{0}^{\infty}(M)$, then $(1-\theta)u\in C_{0}^{\infty}(M\backslash K)$. Hence
    \begin{equation}\label{ineq1}
    \begin{aligned}
        ||(1-\theta)u||_{L^{\frac{2p}{p-2}}(M\backslash K)}
        &\leq C_{1} ||d((1-\theta)u)||_{L^{2}(M\backslash K)}\\
        &\leq C_{1}(||d\theta||_{L^{\infty}(\tilde{K})}||u||_{L^2(\tilde{K})}+||du||_{L^2(M)}).
    \end{aligned}
    \end{equation}
    We also notice Sobolev inequality on $\tilde{K}$
    \begin{equation}\label{ineq2}
        ||u-\frac{\int_{\tilde{K}}u\ dx}{Vol(\tilde{K})}||_{L^{\frac{2p}{p-2}}(\tilde{K})} \leq C_{2}||du||_{L^2(\tilde{K})}
    \end{equation}
    and the Poincar\'e inequality on $\tilde{K}$ implies
    \begin{equation}\label{ineq3}
    \begin{aligned}
        ||u||_{L^2(\tilde{K})} \leq C_{3}||du||_{L^2(\tilde{K})} + \frac{\left|\int_{\tilde{K}}u\ dx\right|}{(Vol(\tilde{K}))^{\frac{1}{2}}}.
    \end{aligned}    
    \end{equation}
    Combine \eqref{ineq1}, \eqref{ineq2} and \eqref{ineq3}, we have
    \begin{equation}
    \begin{aligned}
        &||u||_{L^{\frac{2p}{p-2}}(M)}\\ 
        &\leq ||(1-\theta)u||_{L^\frac{2p}{p-2}(M)} + ||\theta (u-\frac{\int_{\tilde{K}}u\ dx}{Vol(\tilde{K})})||_{L^\frac{2p}{p-2}(M)} + 
        ||\theta \frac{\int_{\tilde{K}}u\ dx}{Vol(\tilde{K})}||_{L^\frac{2p}{p-2}(M)}\\
        &\leq ||(1-\theta)u||_{L^\frac{2p}{p-2}(M\backslash K)} + ||(u-\frac{\int_{\tilde{K}}u\ dx}{Vol(\tilde{K})})||_{L^\frac{2p}{p-2}(\tilde{K})} + 
        ||\frac{\int_{\tilde{K}}u\ dx}{Vol(\tilde{K})}||_{L^\frac{2p}{p-2}(\tilde{K})}\\
        &\leq C_{4}(||du||_{L^2(M)}+|\int_{\tilde{K}}u \ dx|).
    \end{aligned}
    \end{equation}
    It follows from \cite[Lemma 3.10]{li-harmonic} that $M$ is nonparabolic. According to \cite{Ancona} and \cite{carron_L^2_cohomologie}, we have
    \begin{equation}
        \left| \int_{\tilde{K}}u\ dx \right| \leq C_{5}\Big(\int_{M}|du|^2\ dx\Big)^{\frac{1}{2}}.
    \end{equation}
    Finally, we obtain the conclusion.
\end{proof}

\bibliographystyle{alpha}

\bibliography{references}
\end{document}